\documentclass[a4paper, reqno, 11pt]{amsart}
\makeatletter

\@addtoreset{equation}{section}
\makeatother
\usepackage{enumerate}
\usepackage{setspace}
\usepackage{color}
\usepackage{amsmath}
\usepackage{amssymb}
\usepackage{latexsym}
\usepackage{amsthm}
\usepackage{textcomp}
\usepackage{hyperref}
\usepackage{here}
\newtheorem{Thm}{Theorem}[section]

\newtheorem{Lem}[Thm]{Lemma}
\newtheorem{Prop}[Thm]{Proposition}

\newtheorem{Cor}[Thm]{Corollary}

\theoremstyle{definition}

\newtheorem{Def}[Thm]{Definition}
\newtheorem{remark}[Thm]{Remark}
\newtheorem{Exa}[Thm]{Example}

\newcommand{\R}{{\mathbb{R}}}
\newcommand{\C}{{\mathbb{C}}}
\newcommand{\D}{{\mathbb{D}}}
\newcommand{\Q}{{\mathbb{Q}}}
\newcommand{\N}{{\mathbb{N}}}
\renewcommand{\H}{{\mathbb{H}}}
\let\Mob\phi
\let\z z
\begin{document}

\title[]{Characterizations of the maximum likelihood estimator of the Cauchy distribution}

\author[K. Okamura]{Kazuki Okamura}
\address{Department of Mathematics, Faculty of Science, Shizuoka University}
\email{okamura.kazuki@shizuoka.ac.jp}

\author[Y. Otobe]{Yoshiki Otobe}
\address{Department of Mathematics, Faculty of Science, Shinshu University}
\email{otobe@math.shinshu-u.ac.jp}

\dedicatory{}

\maketitle

\begin{abstract}
This paper gives a new approach for the maximum likelihood estimation of the joint of the location and scale of the Cauchy distribution. We regard the joint as a single complex parameter and derive a new form of the likelihood equation of a complex variable. Based on the equation, we provide a new iterative scheme approximating the maximum likelihood estimate.  We also handle the equation in an algebraic manner and derive a polynomial containing the maximum likelihood estimate as a root. This algebraic approach provides another scheme approximating  the maximum likelihood estimate by root-finding algorithms for polynomials, and furthermore, gives non-existence of closed-form formulae for the case that the sample size is five. We finally provide some numerical examples to show our method is effective.
\end{abstract}

\section{Introduction}\label{sec:introduction}

In both theoretical and practical fields in statistical studies,
the Gaussian distribution is frequently used because it is mathematically
understood very well.
However, it is not always clear why the sample obey the Gaussian law.
For example,
it is hardly ever possible to guarantee that hypotheses for the central limit
theorem are fulfilled in the real world.
It is also worthy to point out that Gaussian samples almost never allow the
existence of outliers.
One possible way to overcome this state of affairs,
and that we propose in the present paper, is to use the
Cauchy distribution instead.
The Cauchy distribution as well as the Gaussian is stable and
bell-shaped distribution that has a probability density function
controlled by two parameters: the location parameter $\mu \in \R$
and the scale parameter $\sigma > 0$.
Moreover, samples from the Cauchy distribution almost concentrate
in an area around the location $\mu$ except outliers. 
So it is natural to build statistical models using the Cauchy distribution.
But unfortunately, at least from a mathematical point of view,
it also yields a sensitive issue, 
which is the Cauchy distribution is not integrable.
Because of such weak integrability,
it is hard to handle it in a probability theoretic manner. 
Only few theories of the unbiased estimators for the Cauchy distribution have been known,
see e.g., \cite{McCullagh1993} and \cite{Rothenberg1964}.
Let us mention that, in this direction, a series of articles
\cite{Akaoka2021-2,Akaoka2021-3,Akaoka2021-1}
introduced and analyzed unbiased estimators
for the Cauchy distributions.

Another well-studied estimator is the maximum likelihood estimator  based on the probability density functions. 
However, 
the maximal likelihood estimator for the Cauchy distribution 
is not easy to analyze in the case that neither the location nor the scale is known. 
The existence and uniqueness of
the maximum likelihood estimate was shown by Copas \cite{Copas1975}. 
Ferguson \cite{Ferguson1978} gave closed-form formulae for both of the location and scale parameters
$\mu$ and $\sigma$ by ordering the sample, when the size of the sample was three or four. 
But the derivation of the formulae was omitted there. 
Later McCullagh \cite{McCullagh1996} gave derivations of them by geometric considerations. 
To our knowledge, the existence or nonexistence of explicit formulae has been known when the sample size is bigger than four.

Under the circumstances, it is interesting to find nice numerical techniques to find the maximum likelihood estimates of samples from  the Cauchy distribution.
Haas, Bain and Antle \cite{Haas1970} and Hinkley \cite{Hinkley1978} used the Newton--Raphson method to compute the maximum likelihood estimates.
We note that, though the Newton--Raphson method works well in many cases,
the convergence is not guaranteed in general and the algorithm could diverge. 

The purpose of the present paper is to broaden and deepen these studies. 
We will introduce a new iterative scheme (dynamical system) for the maximum likelihood estimate of the joint of the location and scale parameters, which converges from every starting point in the upper-half plane. 
The argument is not only simple but asserts that the estimate is stable with respect to small perturbations of the sample. 
We will describe them in Section \ref{sec:fixed}.
The argument  there is also applicable to the wrapped (circular) Cauchy distribution on the circle, 
which is an easier and different approach from Kent and Tyler's \cite{Kent1988}. 

In Section \ref{sec:poly}, we also provide  a polynomial  one of which roots is the  maximum likelihood estimate of the joint of the location and scale parameters.
By using the polynomial, we can derive Ferguson's formulae \cite{Ferguson1978} easily
when the sample size is three or four, 
and moreover we will analyze its algebraic structure, specifically the Galois group, 
when the sample size is bigger than four. 
Our method is largely different from McCullagh's one \cite{McCullagh1996}. 
With the help of a computer algebra system,
it can be proved there are no algebraic closed form formulae for the maximal likelihood estimates
when $n = 5, 6$ and $7$. 
We conjecture that 
there is no algebraic representation of the maximum likelihood estimator
when the size of the sample is bigger than $4$.

The key idea is to regard the location and scale parameters of the Cauchy distribution as a single complex parameter;
$\theta := \mu + i \sigma \in \H$, 
where $\H$ is the upper-half complex plane.
Such parameterization was first considered by McCullagh \cite{McCullagh1996}
in the context of invariance of the law of Cauchy random variables under the M\"obius transformations. 
We begin our discussions with the likelihood equations of the Cauchy distribution under the complex parametrization 
in Section \ref{sec:framework}. 
It enables us to derive some new probabilistic
properties concerning the maximum likelihood estimate of the Cauchy distribution,
which are summarized in Section \ref{sec:properties}.

Our results are applicable to
numerical computations for the maximum likelihood estimates. 
Two approximation methods have been developed in Sections \ref{sec:fixed} and \ref{sec:poly}. 
One is the iterative scheme appearing in Section \ref{sec:fixed} and the other is to apply root-finding algorithms for polynomials to the polynomial containing the maximum likelihood estimate as a root appearing  in Section \ref{sec:poly}. 
We emphasize that those work well even for samples
for which numerical approximation schemes such as the Newton-Raphson method or the Nelder-Mead method  fail.

\section{Complex parametrization of the likelihood equations}\label{sec:framework}

We are given a real sequence of independent observed sample
$\{x_1, x_2, \ldots, x_n\}$, $n \geq 3$, from a Cauchy distribution
$C(\mu, \sigma)$ with a scale parameter $\mu \in \R$ and
a location parameter $\sigma > 0$; both are unknown.
Here, we say a random variable $X$ obeys a Cauchy distribution $C(\mu, \sigma) \equiv C(\theta)$, $\theta = \mu + i \sigma$,
if the probability density function of $X$ is given by
\begin{equation}\label{eq:Cauchy-density}
f(x; \mu, \sigma) \equiv f(x; \theta) = 
\frac{\sigma}{\pi} \frac{1}{(x-\mu)^2 + \sigma^2}
= \frac{1}{2\pi i} \left( \frac{1}{x - \theta} - \frac{1}{x - \overline{\theta}} \right),
\end{equation}
where $i$ denotes an imaginary unit and $\overline{\theta}$ denotes
the complex conjugate of $\theta \in \C$, where $\C$ is the complex plane.
We will denote by $\Re(\theta)$ and $\Im(\theta)$ the
real part and the imaginary part, respectively,
that is, $\Re(\theta) = \mu$ and $\Im(\theta) = \sigma$.
Then, the likelihood function for this sample is defined by
\begin{equation}\label{eq:likelihood-function}
L(\theta) \equiv L(\theta; x_1, x_2, \ldots, x_n)
:= \prod_{j=1}^n f(x_j; \theta).
\end{equation}

Our aim in the present paper is to consider the maximum likelihood estimate 
$\hat{\theta}$ for $\{x_1, x_2, \ldots, x_n\}$, which is the maximizer of
the likelihood function, defined through
\begin{equation}\label{eq:def-MLE}
\hat{\theta} \equiv
\hat{\theta}(x_1, x_2, \ldots, x_n)
:= \mathrm{arg}\,\max_{\theta \in \C} \{L(\theta;x_1,x_2,\ldots,x_n)\}.
\end{equation}
Copas \cite{Copas1975} showed that the maximizer for the likelihood 
function of the Cauchy distribution is unique if $n \geq 3$;
hence $\hat{\theta}$ is well-defined. 

\begin{remark}
Introducing a function $\delta_x(\theta) := \frac{1}{2\pi i}\frac{1}{x - \theta}$, the right hand side of \eqref{eq:Cauchy-density} becomes
$\delta_x(\theta) + \overline{\delta_x(\theta)}$.
It is from this representation clear that \eqref{eq:Cauchy-density}
is real valued.
\end{remark}

As we already mentioned in the introduction, 
our main idea is to
regard the parameters of the Cauchy distribution as a single complex
number. 
It means that
we will find the maximizer $\hat{\theta}$ 
of \eqref{eq:likelihood-function} in the
upper-half plane $\H := \{\theta \in \C;
\Im(\theta) > 0\}$ of the complex plane,
while the likelihood function $L(\theta)$ is a real-valued
function defined on $\H$.
In this manner, since $\Re(\theta) = \mu$ and $\Im(\theta) = \sigma$, 
we will automatically
find both the scale and location parameters simultaneously.

To derive the likelihood equation
to which the maximizer of \eqref{eq:likelihood-function} is the solution,
it is convenient to introduce the following
polynomial which will play a central role in the present paper:
\begin{equation}\label{eq:def-p}
h(\theta) \equiv h(\theta; x_1, x_2, \ldots, x_n) := \prod_{j=1}^n (x_j - \theta).
\end{equation}

We will find a stationary point of the likelihood function \eqref{eq:likelihood-function} with respect to $\theta$.

\begin{Prop}
The maximum likelihood estimate $\hat{\theta}$ for
the observed sample $x = (x_1, x_2, \ldots, x_n)$, each obeying a Cauchy
distribution independently, 
is the solution to the following likelihood equation on $\H$:
\begin{equation}\label{eq:Cauchy-likelihood-equation}
nh(\theta) - (\theta - \overline{\theta}) h'(\theta) = 0.
\end{equation}
\end{Prop}

\begin{proof}
We regard the likelihood function $L(\theta)$ of \eqref{eq:likelihood-function}
as a function of $\theta$ and $\overline{\theta}$.
Then it is sufficient (see \cite{Messerschmitt:EECS-2006-93})
for stationary points of $L(\theta, \overline{\theta};x)$,
a real valued function on the complex plane,
to find the zeros of
\begin{multline*}
\frac{\partial L}{\partial \theta}(\theta, \overline{\theta};x)
= \left( \frac{1}{2 \pi i} \right)^n
\sum_{j=1}^n \frac{1}{(x_j - \theta)^2} \prod_{k \neq j} \left( \frac{1}{x_k - \theta} - \frac{1}{x_k - \overline{\theta}} \right) \\
= \left(\frac{1}{2 \pi i}\right)^n
\sum_{j=1}^n \frac{1}{(x_j - \theta)^2} \prod_{k \neq j} \left( \frac{\theta - \overline{\theta}}{(x_k - \theta)(x_k - \overline{\theta})}  \right)
\end{multline*}
on $\H$.
Multiplying both hand sides of above 
by $\prod_{l=1}^n (x_l - \theta)^2 (x_l - \overline{\theta})$
that takes zero at $\theta = x_l$ ($l = 1, 2, \ldots, n$) only,
the right hand side becomes

\begin{equation*}
\prod_{l=1}^n (x_l - \theta)^2 (x_l - \overline{\theta})
\sum_{j=1}^n \frac{1}{(x_j - \theta)^2} \prod_{k \neq j} \frac{\theta - \overline{\theta}}{(x_k - \theta)(x_k - \overline{\theta})}
=
(\theta - \overline{\theta})^{n-1} \sum_{j=1}^n (x_j - \overline{\theta}) \prod_{l \neq j} (x_l - \theta).
\end{equation*}
Since $x_j - \overline{\theta} = (x_j - \theta) + (\theta - \overline{\theta})$, 
this equals to
$(\theta - \overline{\theta})^{n-1} \left( nh(\theta) - (\theta - \overline{\theta}) h'(\theta) \right)$, 
which concludes that $\frac{\partial L}{\partial \theta}(\theta, \overline{\theta};x) = 0$ if and only if \eqref{eq:Cauchy-likelihood-equation} holds.
\end{proof}

\begin{remark}\label{rem:symmetric polynomials}
Note that, using elementary symmetric polynomials $s_0 = 1$,
$s_1 := \sum_{j=1}^n x_j$,
$s_2 := \sum_{i < j}^n x_i x_j$, \ldots, $s_n := x_1 x_2 \ldots x_n$,
we see that 
\begin{equation}\label{eq:p-polynomial}
h(\theta) = \sum_{j=0}^n s_{n-j} (-\theta)^j = s_n - s_{n-1}\theta + s_{n-2} \theta^2 + \dots 
+ (-1)^{n-1} s_1 \theta^{n-1} + (-1)^n s_0 \theta^n.
\end{equation}
And the Leibniz rule leads us to
\begin{equation}\label{eq:h'}
h'(\theta) = - \sum_{j=1}^n \prod_{i \neq j} (x_j - \theta)
= -\sum_{j=1}^{n} j s_{n-j}(-\theta)^{j-1}. \qedhere
\end{equation}
\end{remark}

\begin{remark}
It is clear that \eqref{eq:Cauchy-likelihood-equation} holds
at every $\theta = x_j$ ($j = 1, 2, \ldots, n$).
Though $h(\theta) \equiv h(\theta; x_1, x_2, \ldots, x_n)$ has degree at most $n$ in $\theta$, 
\eqref{eq:Cauchy-likelihood-equation} has $n + 2$ roots in $\C$
by virtue of Copas' result.
\end{remark}

It is clear that $h(\theta) = 0$ if and only if $\theta = x_j$ ($j = 1, 2, \ldots, n)$.
From Rolle's theorem
(or Gauss--Lucas' theorem), 
it means that all zeroes of $h'(\theta)$ are
real and separated by these roots of $h(\theta) = 0$.

Now we assume that all observed sample $x_1, x_2, \ldots, x_n$ are
distinct.
It occurs almost surely because we assume that they obey a Cauchy distribution
independently.
Hence, we see that  $h(\theta) = 0$ and $h'(\theta) = 0$ never occur
simultaneously.  
Therefore \eqref{eq:Cauchy-likelihood-equation} is equivalent to
\begin{equation}\label{eq:Cauchy-likelihood-equation2}
\overline{\theta} = \theta - n \frac{h(\theta)}{h'(\theta)}
\end{equation}
to which a solution in the upper half complex plane $\H$ is
the maximum likelihood estimate $\hat{\theta}$ of
$\{x_1, x_2, \ldots, x_n\}$.
We also note that \eqref{eq:Cauchy-likelihood-equation2} is equivalent to
\begin{equation}\label{eq:Cauchy-likelihood-equation3}
\frac{1}{n} \sum_{j=1}^n \frac{1}{x_j - \theta} = \frac{1}{\overline{\theta} - \theta}.
\end{equation}

Thus, we have, comparing the real part and the imaginary part respectively,
the following form of the likelihood equation (See, e.g., \cite{Copas1975, Ferguson1978}) for the location parameter and the scale parameter.

\begin{Cor}
The real part $\hat{\mu}$ and the imaginary part $\hat{\sigma}$
of the maximum likelihood estimate $\hat{\theta}$ of $\{x_1, \dots, x_n\}$ solve
\begin{equation}\label{eq:classical-likelihood-equation}
\begin{cases}
\sum\limits_{j=1}^n \frac{x_j - \mu}{(x_j - \mu)^2 + \sigma^2} = 0, \\
\sum\limits_{j=1}^n \frac{\sigma^2}{(x_j-\mu)^2 + \sigma^2} = \frac{n}{2}.
\end{cases}
\end{equation}
\end{Cor}

Comparing \eqref{eq:classical-likelihood-equation} with
\eqref{eq:Cauchy-likelihood-equation3} or
\eqref{eq:Cauchy-likelihood-equation4} below
clarifies the advantage to formulate the maximum likelihood estimator
by a single complex parameter.
Therefore, in the sequel of the paper, 
we will concentrate on finding $\hat{\theta}$ satisfying \eqref{eq:Cauchy-likelihood-equation2}.

We begin with analyzing the right hand side of
\eqref{eq:Cauchy-likelihood-equation2}.
Let 
$$q(\theta) := \theta - n \frac{h(\theta)}{h'(\theta)}.$$
Let us introduce the following
M\"obius transformation (see \cite{Rudin1987}):
\begin{equation}
\Mob_\theta(\zeta) := \frac{\zeta - \theta}{\zeta - \overline{\theta}}.
\end{equation}

Note that $\Mob_\theta$ is a bijective holomorphic map
from the upper half plane $\H$ to the unit open disc centered at origin $\D$.
Then, $q(\theta)$ can be expressed as follows:

\begin{Lem}
For $\theta \in \H$ and $x_1, x_2, \ldots, x_n \in \R$, we see that 
\begin{equation}\label{eq:Mobius-representation}
q(\theta) = \overline{ \Mob_\theta^{-1}\left( \frac{1}{n} \sum_{j=1}^n \Mob_\theta(x_j) \right)}.
\end{equation}
\end{Lem}

\begin{proof}
First, we note that $q(\theta) = \frac{\sum\limits_{j=1}^n \frac{x_j}{x_j - \theta}}{\sum\limits_{j=1}^n \frac{1}{x_j - \theta}}$ and
$\Mob_\theta^{-1}(\zeta) = \frac{\theta -\overline{\theta}\zeta}{1 - \zeta}$.
Let us compute the numerator of the right hand side of \eqref{eq:Mobius-representation} without taking complex conjugate. 
\[
\theta -\overline{\theta}\frac{1}{n}\sum_{j=1}^n \Mob_\theta(x_j)
= \frac{1}{n} \theta \sum_{j=1}^n \frac{x_j - \overline{\theta}}{x_j - \overline{\theta}}-\overline{\theta}\frac{1}{n} \sum_{j=1}^n \frac{x_j - \theta}{x_j - \overline{\theta}} = \frac{\theta - \overline{\theta}}{n} \sum_{j=1}^n \frac{x_j}{x_j - \overline{\theta}}.
\]
The denominator of \eqref{eq:Mobius-representation} can be computed as
\[
1 - \frac{1}{n} \sum_{j=1}^n \frac{x_j - \theta}{x_j - \overline{\theta}} = \frac{1}{n} \sum_{j=1}^n \frac{x_j - \overline{\theta}}{x_j - \overline{\theta}}
- \frac{1}{n} \sum_{j=1}^n \frac{x_j - \theta}{x_j - \overline{\theta}} = \frac{\theta - \overline{\theta}}{n} \sum_{j=1}^n \frac{1}{x_j - \overline{\theta}}.
\]
Since $\theta \in \H$ and
$x_1, x_2, \ldots, x_n \in \R$, we have the conclusion.
\end{proof}

\begin{remark}
In general, a quantity like
$\Mob^{-1}\left( \frac{1}{n} \sum_{j=1}^n \Mob(x_j) \right)$ is called
a quasi-arithmetic mean of $\{x_1, x_2, \ldots, x_n\}$
with generator $\Mob$.
In this point of view, 
the maximum likelihood estimate $\hat{\theta}$
of the
observed sample is a
complex conjugate of a quasi-arithmetic mean
generated by a M\"obius transformation,
while the generator itself depends on the parameter.
Complex-valued quasi-arithmetic means are essential tools in statistical analysis of the
Cauchy distributions.
Indeed, they often give unbiased estimators for the parameter.
See \cite{Akaoka2021-1} for limit theorems for complex-valued quasi-arithmetic mean of random variables and its application to unbiased estimations of the Cauchy distribution.
We emphasize that the effectiveness of the 
M\"obius transformations to study Cauchy distributions were
first examined by McCullagh \cite{McCullagh1996}.
\end{remark}

\begin{Cor}\label{cor:Cauchy-likelihood-equation4}
The maximum likelihood estimate $\hat{\theta}$ of $\{x_1, \dots, x_n\}$ solves
\begin{equation}\label{eq:Cauchy-likelihood-equation4}
\sum_{j=1}^n \Mob_\theta(x_j) = 
\sum_{j=1}^n \frac{x_j - \theta}{x_j - \overline{\theta}} = 0.
\end{equation}
\end{Cor}

This is also another form of the likelihood equation in the complex form. 
If we let $\tilde h (z,w)  := \sum_{j=1}^n \frac{x_j - z}{x_j - w}$, then, \eqref{eq:Cauchy-likelihood-equation4} is equivalent to the equation $\tilde h (\theta, \overline{\theta}) =0$. 
This plays an important role when we consider the Bahadur efficiency of the maximum likelihood estimator (\cite{Akaoka2021-2}). 

\begin{proof}
From the above lemma, we see that 
$\Mob_\theta( \overline{q(\theta)} ) = \frac{1}{n} \sum_{j=1}^n \Mob_\theta(x_j)$.
The assertion follows from $\overline{q(\theta)} = \theta$ and $\Mob_\theta(\theta) = 0$.
\end{proof}

\begin{Prop}
$q(\cdot)$ maps the upper and lower half plane to the lower and upper half plane,
respectively.
\end{Prop}

\begin{proof}
Since each
$\Mob_\theta(x_j)$, $j = 1, 2, \ldots, n$,
is on the unit circle centered at the origin for every
$x_j \in \R$ and $\theta \in \H$,
$\frac{1}{n} \sum_{j=1}^{n} \Mob_\theta (x_j)$ is contained in the
unit disc centered at the origin.
The assertion now follows from the expression \eqref{eq:Mobius-representation}.
\end{proof}

Therefore we can restrict our consideration to the upper half complex
plane by setting $Q(\theta) := q(q(\theta))$.
Then $Q : \H \to \H$ is holomorphic,
and, by \eqref{eq:Cauchy-likelihood-equation2}, the maximum likelihood estimator $\hat{\theta}$ satisfies
$\hat{\theta} = Q(\hat{\theta})$. 
Considerations of the maximal likelihood estimates in the following sections depend on this equality. 

\section{Iterative scheme for the maximum likelihood estimates}\label{sec:fixed}

As is just stated in the preceding section,
the maximum likelihood estimate $\hat{\theta}$ for the Cauchy distribution
is a solution to $\theta = Q(\theta)$.

We first study an approximating scheme to achieve the solution.
We show that this scheme is effectively applicable to the numeric
computation for the maximum likelihood estimates,
while some examples will be discussed in Appendix \ref{sec:numeric}.
Then we will show that our scheme also works
in the case of the circular Cauchy distributions.

\subsection{Fixed point of $Q$}

The maximum likelihood estimate
$\hat{\theta}$ of an observed sample $\{x_1, x_2, \ldots, x_n\}$ from
a Cauchy distribution is a fixed point of $Q: \H \to \H$,
where $Q(\theta) = q(q(\theta))$, and $q(\theta)$ is 
expressed as \eqref{eq:Mobius-representation}.
Our strategy to achieve $\hat{\theta}$ is 
to construct an orbit $\z_0, \z_1, \ldots$
defined by $z_m = Q(\z_{m-1})$.
If it has a limit point $\z_\infty$, it satisfies
$\z_\infty = Q(\z_\infty)$.

\begin{remark}
Recall that 
$Q$ depends on the observed sample
$\{x_1, x_2, \ldots, x_n\}$, since $q(\cdot)$ does.
Because the observed samples are realization of independent
random variables $X_1, X_2, \ldots, X_n$ obeying the Cauchy distribution,
$Q$ can be seen as a random holomorphic function.
In this sense, our iterative scheme is a random dynamical system.
Therefore the maximum likelihood estimator $\hat{\theta}(X_1, X_2, \ldots, X_n)$ is a random fixed point of the random dynamics $Q$. 
We note that the law of  $\hat{\theta}(X_1, X_2, \ldots, X_n)$ 
was studied by McCullagh \cite{McCullagh1993}.
\end{remark}

To assert the unique existence of the fixed point,
Banach's fixed point theorem is often used.
But in our situation, $Q$ is not contractive nor $\H$ is not complete.
Moreover, $Q$ clearly has some
fixed points on the boundary of $\H$ (real axis).
But we will show that
similar statements hold and every orbit defined by $z_n := Q(z_{n-1})$
starting at every point in $\H$ converges to a unique attracting fixed point.
It also recovers Copas' uniqueness result for the maximum likelihood
estimate \cite{Copas1975},
but we believe the argument presented here, which is complex analytic,
is much simpler than his. 
Furthermore, this iterative scheme gives
an efficient numerical way to compute the maximum likelihood estimate.

Recall that we assume that the sample size $n \geq 3$
and the observed sample $x_1, x_2, \ldots, x_n$ are distinct.

\begin{Thm}\label{main}
(i) The equation $\z = Q(\z)$ has a unique solution in $\H$
	which is equal to the maximum likelihood estimate $\hat{\theta}(x_1, \dots, x_n)$.\\
(ii) The unique solution $\hat{\theta}$ of the equation $\z = Q(\z)$ in $\H$
	is a fixed point of $Q$.\\
(iii) For every $\z \in \H$, $\lim_{m \to \infty} Q^m (\z) = \hat{\theta},$ 
where $Q^m(\z) := \overbrace{Q(Q(\dots Q}^m(z)))$ denotes the $m$-times iteration of $Q$. 
This convergence is exponentially fast. 
\end{Thm}

As a corollary, it is easy to prove 
that the maximum likelihood estimate is stable with respect to
small perturbation of the observed sample. 

\begin{Cor}\label{stability}
Denote the maximum likelihood estimate of 
$x_1^{(m)}, \dots, x_n^{(m)}$
from the Cauchy distribution by $\hat{\theta}^{(m)}$,
for each $m \in \N \cup \{\infty\}$. 
If the $n$-dimensional vectors
$\left(x_1^{(m)}, \dots, x_n^{(m)}\right)$
converges to $\left(x_1^{(\infty)}, \dots, x_n^{(\infty)}\right)$
as $m \to \infty$, 
then, $\hat{\theta}^{(m)} \to \hat{\theta}^{(\infty)}$ as $m \to \infty$.
\end{Cor}

This corollary is obvious when $n = 3$ or $4$ from Ferguson's 
explicit formulae for the maximum likelihood estimates \cite{Ferguson1978}.
However, we will see in Section \ref{sec:poly} below,
we could not expect to exist algebraic closed-form formulae for $n \ge 5$.

\begin{proof}[Proof of Theorem \ref{main}]
Since $n \ge 3$, 
we see that for every $\epsilon > 0$ there exists a compact subset $K \subset \H$ such that 
\begin{equation}\label{cpt}
\sup_{\theta \in \H \setminus K} L(\theta; x_1, \dots, x_n) <  \epsilon.
\end{equation}

We first show \eqref{cpt}. 
Set $\theta = \mu + \sigma i$. 
Since 
\[ \frac{\sigma}{\pi} \frac{1}{(x_j - \mu)^2 + \sigma^2} \le \frac{1}{\pi} \min\left\{ \frac{1}{\sigma}, \frac{\sigma}{(x_j-\mu)^2}, \frac{1}{|x_j - \mu|}  \right\},\]
there exists $M_1$ such that if $|\mu| > M_1$ or $\sigma > M_1$, 
then, for each $j$, 
$\frac{\sigma}{\pi} \frac{1}{(x_j - \mu)^2 + \sigma^2} \le \epsilon^{1/n}$.

Let $\delta_0 > 0$ such that $[x_i - \delta_0, x_i + \delta_0] \cap  [x_j - \delta_0, x_j + \delta_0] = \emptyset$ for $i \ne j$.  
If $\mu \in [x_i - \delta_0, x_i + \delta_0]$ for some $i$, 
then, 
$L(\theta; x_1, \dots, x_n) \le \frac{\sigma^{n-2}}{\pi \delta_0^{2(n-1)}}$.
If $\mu \notin [x_i - \delta_0, x_i + \delta_0]$ for every $i$, 
then, for each $j$, 
$\frac{\sigma}{\pi} \frac{1}{(x_j - \mu)^2 + \sigma^2} \le  \frac{\sigma}{\pi \delta_0^2}$.

Since $n \ge 3$, there exists $M_2 > M_1$ such that if either $|\mu| > M_2$, $\sigma > M_2$, or $|\mu| \le M_2$ and $\sigma < 1/M_2$, 
then, $L(\theta; x_1, \dots, x_n) \le \epsilon$.
Now \eqref{cpt} follows if we let $K := \left\{\mu + \sigma i : |\mu| \le M, \sigma \in [1/M, M] \right\}$. 
Thus we see that  there exists at least one element in $\H$ which attains the maximum of $L(\theta; x_1, \dots, x_n)$. 

Let $\z_0$ be an  element in $\H$ which attains the maximum of $L(\theta; x_1, \dots, x_n)$. 
Then, $\overline{\z_0} = q(\z_0)$. 
Consider the map $\Mob_{\z_0} \circ Q \circ \Mob_{\z_0}^{-1}$. 
Since $Q(\z_0) = \z_0$, 
this is a bijective holomorphic map from $\D$ to $\D$ such that $\Mob_{\z_0} \circ Q \circ \Mob_{\z_0}^{-1} (0) = 0$. 
Now by Schwarz's lemma \cite{Rudin1987}, 
we see that 
$\left| Q'(\z_0) \right| = \left| \left(  \Mob_{\z_0} \circ Q \circ \Mob_{\z_0}^{-1} \right)'(0) \right| \le 1$. 
We will show that $\left| Q'(\z_0) \right|  < 1$. 
Let us assume that $ \left| Q'(\z_0) \right| = 1$.
Then, for some $a \in \C$ with $|a| = 1$, 
$\Mob_{\z_0} \circ Q \circ \Mob_{\z_0}^{-1} (w) = aw$ for every $w \in \D$, and hence 
\[Q(\z) = \Mob^{-1}_{\z_0} \left(a \Mob_{\z_0}(\z)\right)
= \frac{(\z_0 - a \overline{\z_0}) \z + (1+a) |\z_0|^2}{(1-a)\z + a\z_0 - \overline{\z_0}}, \ \z \in \H.\]
This cannot occur if $n \ge 3$. 
See also Remark \ref{Rmk-Q} (i) below.
Hence, we see that 
$\left| Q'(\z_0) \right|  < 1$. 
See Remark \ref{Rmk-Q} (ii) below for an alternative proof of this inequality. 
Now by Schwarz's lemma, we see that $\left|\Mob_{\z_0} \circ Q \circ \Mob_{\z_0}^{-1}(w)\right| < |w|$ for every $w \in \D$ and $w \ne 0$. 
In particular there is no fixed point of $Q$ in $\H$ other than $\z_0$. 
This completes the proofs of assertions  (i) and (ii). 

Now we proceed to the proof of assertion (iii). 
For ease of notation, we set
$F := \Mob_{\z_0} \circ Q \circ \Mob_{\z_0}^{-1}$. 
Let $\widetilde \z \in \H$ and $\widetilde w := \Mob_{\z_0}(\widetilde \z)$.
Then, since $\Mob_{\z_0}$ is bijective, 
in order to show that $\lim_{m \to \infty} Q^m (\widetilde \z) = \z_0$, 
it suffices to show that $\lim_{m \to \infty} F^m (\widetilde w) = 0$. 
Let 
\[ G(w) := \begin{cases} \frac{F(w)}{w}, & w \ne 0,\\
 F'(0), &  w = 0.\end{cases} \]
 
It is easy to see that $F'(0) = Q'(\z_0)$.
Since $F(0) = 0$ and $F$ is holomorphic on $\D$, $G$ is also holomorphic on $\D$. 
Let 
$r := \max_{w \in \C : |w| \le |\widetilde w| } \left|G(w)\right|. $
Since $|F(w)| < |w|$ for every $w \in \D$ as in the proof of (ii) above, we see that $r < 1$. 
It holds that $|F(w)| \le r |w|$, for $w \in \D$ such that $|w| \le |\widetilde w|$.
Hence, $|F^n (\widetilde w)| \le r^m |\widetilde w| \le r^m \to 0$, $m \to \infty$.
Since 
\[ \frac{|Q^m (\widetilde \z) - \z_0|}{|Q^m (\widetilde \z) - \overline{\z_0}|} \le r^m  \left|\frac{\widetilde \z - \z_0}{\widetilde \z - \overline{\z_0}} \right| \le r^m, \]
the convergence of $(Q^m (\widetilde \z))_m$ to $\z_0$ is exponentially fast. 
\end{proof}

\begin{remark}\label{rem:may-be-slow}
Let $r$ be the constant in the proof of Theorem \ref{main}. 
Then, by the maximum principle, 
\begin{align*}
r &= \sup\left\{ \left|\frac{Q (\z) - \z_0}{Q ( \z) - \overline{\z_0}}\right| / \left|\frac{\z - \z_0}{\z - \overline{\z_0}}\right| :  \left|\frac{\z - \z_0}{\z - \overline{\z_0}}\right| \le  \left|\frac{\widetilde \z - \z_0}{\widetilde \z - \overline{\z_0}}\right|, \ \z \ne \z_0 \right\} \\
&= \max\left\{ \left|\frac{Q (\z) - \z_0}{Q ( \z) - \overline{\z_0}}\right| / \left|\frac{\z - \z_0}{\z - \overline{\z_0}}\right| :  \left|\frac{\z - \z_0}{\z - \overline{\z_0}}\right|  =  \left|\frac{\widetilde \z - \z_0}{\widetilde \z - \overline{\z_0}}\right|\right\}.
\end{align*}
Therefore, we see that
the constant $r$ depends on the sample $(x_1, \dots, x_n)$
and the starting point $\widetilde \z$. 
\end{remark}

\begin{proof}[Proof of Corollary \ref{stability}]
Recall that the function $Q$ depends on the sample.
So let us denote the function $Q$
for $\left(x_1^{(m)}, \dots, x_n^{(m)}\right)$ by $Q_m$,
and set $F_m (\z) := Q_m (\z) - \z$. 
This is holomorphic on $\H$ and $F_m (\z^{(m)}) = 0$. 

By Theorem \ref{main}, 
we see that if we take sufficiently small $\epsilon > 0$, 
then, $F_{\infty}$ has exactly one root in the ball
$B(\z^{(\infty)}, \epsilon)$,
centered at $\z^{(\infty)}$ and the radius is $\epsilon$,
which is $\z^{(\infty)}$. 
We see that 
$\min_{\z; |\z-\z^{(\infty)}| = \epsilon} \left|F_{\infty}(\z)\right| > 0$. 

By the assumption and the fact that each $F_m$ is a rational function, 
$F_m$ converges to $F_{\infty}$ uniformly on the closure of
$B\left(\z^{(\infty)}, \epsilon\right)$ as $m \to \infty$. 
Hence, there exists $N$ such that for every $m \ge N$, 
\[ \max_{\z; |\z-\z^{(\infty)}| = \epsilon} \left|F_{\infty}(\z) - F_m (\z) \right| < \min_{\z; |\z-\z^{(\infty)}| = \epsilon} \left|F_{\infty}(\z)\right|. \]
By Rouch\'e's theorem \cite{Rudin1987}, 
for every $m \ge N$, 
$F_m$ also has exactly one root in the ball $B(\z^{(\infty)}, \epsilon)$.
By  Theorem \ref{main}, it is $\z^{(m)}$. 
Hence $\left|\z^{(m)} - \z^{(\infty)}\right| \le \epsilon$. 
\end{proof}

\begin{remark}\label{Rmk-Q}
(i) There is an alternative proof of Theorem \ref{main} (i). 
The key observation is that $Q$ has at least $n^2 - 2n$ poles on $\R$.   
 We can show that the equation $\z = Q(\z)$ has at least $n^2 - 2n -1$ real solutions by repeated use of the intermediate value theorem. \\
(ii) There is an alternative proof of Theorem \ref{main} (ii). 
We first remark that $q(\z_0) = \overline{\z_0}$. 
Let $x_0$ and $y_0$ be the real and imaginary parts of $\z_0$. 
Then, $Q$ is differentiable at $\z_0$ and 
$Q'(\z_0) =  \left| q'(\z_0) \right|^2$. 
Hence it suffices to show that $\left| q'(\z_0) \right| < 1$. 
We see that for every $\z \in \H$, 
\[ q'(\z) = \frac{\left( \sum_{j=1}^{n} 1/(x_j - \z)\right)^2- n \sum_{j=1}^{n} 1/(x_j - \z)^2 }{\left( \sum_{j=1}^{n} 1/(x_j - \z)\right)^2}.  \]
Using this and $q(\z_0) = \overline{\z_0}$, 
by some calculations, 
we see that 
$|q'(\z_0)| = \frac{1}{n} \left| \sum_{j=1}^{n} \left( \frac{x_j - \overline{\z_0}}{x_j - \z_0} \right)^2\right|$.
Since $| x_j - \overline{\z_0}| = |x_j - \z_0 |, \ 1 \le j \le n$,
and $x_1, \dots x_n$, $n \ge 3$,  are distinctive, 
by the Cauchy--Schwarz inequality, 
we see that $| q'(\z_0) | < 1$.
\end{remark}

\subsection{The circular Cauchy case}\label{sec:circular-Cauchy}

The circular Cauchy distribution, also known as the wrapped Cauchy distribution, appears in the area of directional statistics. 
It is a distribution on the unit circle and is connected with the Cauchy distribution via M\"obius transforms.  
Such connection is considered by  \cite{McCullagh1996}. 
The circular-Cauchy distribution $P^{\mathrm{cc}}_{\psi}$ with parameter $\psi \in \D$ is the continuous distribution on $[0, 2\pi)$  
with density function $\frac{1}{2\pi} \frac{1 - |\psi|^2}{|\exp(ix) - \psi|^2}, \ \ x \in [0, 2\pi)$.

If a random variable $X$ follows the circular-Cauchy distribution $P^{\mathrm{cc}}_{\psi}$, 
then, $\Mob_{\theta}^{-1}(\exp(iX))$ follows the Cauchy distribution with parameter $\Mob_{\theta}^{-1}(\psi)$. 
Hence, for every $\theta \in \H$, 
$\hat{\psi} \in \D$ is the maximum likelihood estimate of $\{x_1, \dots, x_n\} \subset [0, 2\pi)$ from the  circular-Cauchy distribution if and only if $\Mob_{\theta}^{-1}(\hat{\psi}) \in \H$ is that
 of $\{\Mob^{-1}_{\theta}(\exp(i x_1)), \dots, \Mob^{-1}_{\theta}(\exp(i x_n))\}$ from the Cauchy distribution. 
For $\z \in \H$, let 
\[ q_{\theta}(\z) := \frac{\sum_{j=1}^{n} \Mob^{-1}_{\theta}(\exp(i x_j)) / (\Mob^{-1}_{\theta}(\exp(i x_j)) - \z)}{\sum_{j=1}^{n} 1 / (\Mob^{-1}_{\theta}(\exp(i x_j)) - \z)}. \]
and 
let $Q_{\theta}(\z) := q_{\theta}(q_{\theta}(\z))$.  
Then, $Q_{\theta}(\H) \subset \H$, 
and 
$\hat{\psi} \in \D$ is the maximum likelihood estimate of $\{x_1, \dots, x_n\}$ from the circular-Cauchy distribution if and only if 
$\hat{\psi} = \Mob_{\theta} \circ Q_{\theta} \circ \Mob_{\theta}^{-1}(\hat{\psi})$. 
For $w \in \D$, let 
\[ \tilde{q}(w) := \frac{\sum_{j=1}^{n} \exp(i x_j)/(1-w\exp(i x_j))}{\sum_{j=1}^{n} 1/(1- w\exp(i x_j))} \]
and $\tilde{Q}(w) := \tilde{q}(\tilde{q}(w))$.  
We easily see that for every $\theta \in \H$, 
$\tilde{Q}(w) = \Mob_{\theta} \circ Q_{\theta} \circ \Mob_{\theta}^{-1}(w), \ w \in \mathbb{D}$.

By Theorem \ref{main}, we see that 
\begin{Thm}
$\hat{\psi} \in \D$ is the maximum likelihood estimate
of $\{x_1, \ldots, x_n\}$ from the circular-Cauchy distribution
if and only if $\hat{\psi} = \tilde{Q}(\hat{\psi})$. 
Furthermore, for every $w \in \D$, 
$\lim_{m \to \infty} \tilde{Q}^m (w) = \hat{\psi}$.
\end{Thm}

This theorem also gives an algorithm to compute the maximum likelihood estimate which is different from \cite{Kent1988} and \cite{Auslan1995}. 
As we see, the convergence of the algorithm is more easily shown
than \cite{Kent1988} in which an iterative reweighting algorithm
for the maximum likelihood estimation
of the angular Gaussian distribution is considered.

\section{Algebraic approach for the maximum likelihood estimates}\label{sec:poly}

In this section we will handle the likelihood equation
$z = Q(z)$ in $\H$ in an algebraic manner.

Let $P_n$ and $T_n$ be two polynomials  over $\R$ such that 
$\z - Q(\z) = P_n (\z)/T_n (\z)$ 
and the greatest common divisor of $P_n$ and $T_n$ over $\R$ is one.
Recall that $Q(\z)$ depends on the observed sample 
$(x_1, x_2, \ldots, x_n)$.
Then, the degree of $P_n$ is smaller than or equal to $n^2 - 2n + 2$. 
We use the fact that $\z$ is a solution of the equation $\z = Q(\z)$ if and only if $\z$ is a solution of $P_n$. 
We remark that the equation $\z = Q(\z)$ holds for each $z = x_1, x_2, \ldots, x_n$.
Hence, by the cofactor theorem, $R_n (\z) := P_n (\z) / h(\z)$ is a polynomial of degree $n^2 - 3n + 2$ or smaller over $\R$. 
If $\hat{\theta}$ is the maximum likelihood estimate of
a distinctive sample of size $n$ from the Cauchy distribution,
then, $R_n (\hat{\theta}) = 0$, 
in particular,  $\hat{\theta}$ is an algebraic number. 

Ferguson \cite{Ferguson1978} derived closed form formulae for
the solution $\hat{\mu}$ and $\hat{\sigma}$,
$\hat{\theta} = \hat{\mu} + i \hat{\sigma}$,
to the likelihood (simultaneous)
equations \eqref{eq:classical-likelihood-equation}
when $n = 3$ and $4$ under an assumption that the sample is ordered \cite{Ferguson1978}.
But he omitted derivations of them. 
\cite{McCullagh1996} gave derivations of them by using geometric considerations. %
In our algebraic setting, it is easy to derive the
formulae.

\subsection{The case of $n = 3$}

Let $s_j, j = 1,2,3$, be the elementary symmetric polynomials
of samples $\{x_1, x_2, x_3\}$ (See Remark \ref{rem:symmetric polynomials}),
specifically, $s_1 := x_1 +x_2 +x_3, s_2 := x_1 x_2 + x_2 x_3 + x_3 x_1, s_3 := x_1 x_2 x_3$.

Then, $h(\z) = s_3- s_2 \z + s_1 \z^2 - \z^3$. 
We see that 
\[R_3 (\z) =  (3 s_2-s_1^2) \z^2 + (-9s_3 +s_1 s_2) \z+3s_1 s_3 - s_2^2.\]

By solving the equation $R_3 (\z) = 0$, 
we see that the maximum likelihood estimate of $\{x_1, x_2, x_3\}$ is given by 
\[ \hat{\theta} = \frac{s_1 s_2 - 9s_3 +  i \sqrt{-81 s_3^2 + (54 s_1 s_2 - 12 s_1^3)s_3 - 12 s_2^3 + 3 s_1^2 s_2^2} }{2(s_1^2 - 3 s_2)} \]
\[ = \frac{(x_1 + x_2) (x_2 + x_3) (x_3 + x_1) - 8x_1 x_2 x_3 + \sqrt{3} |(x_2 - x_1)(x_3 - x_1)(x_3 - x_2)| i}{2(x_1^2 + x_2^2 + x_3^2 - x_1 x_2 - x_2 x_3 - x_3x_1)}. \]

\subsection{The case of $n = 4$}

As above,  
we let $s_i, i = 1, \dots 4$, be the elementary symmetric polynomials of samples $\{x_1, \dots, x_4\}$, specifically, 
$s_1 := x_1 +x_2 +x_3+x_4, s_2 := x_1 x_2 + x_2 x_3 + x_3 x_4 + x_2 x_3 + x_2 x_4 + x_3 x_4,  s_3 := x_1 x_2 x_3 + x_1 x_3 x_4 + x_1 x_2 x_4 + x_2 x_3 x_4, \ s_4 := x_1x_2x_3x_4$.

Then, we see that 
\begin{multline*}
R_4 (\z) =  (8s_3-4s_1s_2+s_1^3)\z^6
+ (-32s_4-4s_1s_3+8s_2^2-2s_1^2s_2)\z^5 + (40s_1s_4-20s_2s_3+5s_1^2s_3)\z^4 \\
+ (-20s_1^2s_4+20s_3^2)\z^3 + (-40s_3s_4+20s_1s_2s_4-5s_1s_3^2)\z^2 \\
+(32s_4^2+4s_1s_3s_4 -8s_2^2s_4 +2s_2s_3^2)\z
-8s_1s_4^2+4s_2s_3s_4-s_3^3. 
\end{multline*}
Without loss of generality, we can assume that $x_1 < x_2 < x_3 < x_4$. 
Then, $R_4$ is factored as $R_4 (\z) = c F_1 (\z) F_2 (\z) F_3 (\z)$,
where $c$ is a constant and 
\begin{align*}
F_1 (\z) &:= (x_1 - x_2 - x_3 + x_4)\z^2 - 2(x_1x_4 -x_2x_3)\z - x_1x_2x_3 + x_1x_2x_4 + x_1x_3x_4 - x_2x_3x_4, \\
F_2 (\z) &:= (x_1 - x_2 + x_3 - x_4)\z^2 - 2(x_1x_3 - x_2x_4)\z + x_1x_2x_3 - x_1x_2x_4 + x_1x_3x_4 - x_2x_3x_4, \\
F_3 (\z) &:= (x_1 + x_2 - x_3 - x_4)\z^2 - 2(x_1x_2 -x_3x_4)\z + x_1x_2x_3 + x_1x_2x_4 - x_1x_3x_4 - x_2x_3x_4. 
\end{align*}

We denote the discriminants of a quadratic polynomial $F_j$ by $D(F_j)$,
$j = 1, 2, 3$.
They can be computed as
\begin{align*}
D(F_1) &= 4(x_3 - x_4)(x_2 - x_4)(x_1 - x_3)(x_1 - x_2),\\
D(F_2) &= -4(x_3 - x_4)(x_2 - x_3)(x_1 - x_4)(x_1 - x_2),\\
D(F_3) &= 4(x_2 - x_4)(x_2 - x_3)(x_1 - x_4)(x_1 - x_3). 
\end{align*}
Thus we see that  $D(F_2) < 0$, $D(F_1) > 0$ and $D(F_3) > 0$. 

Hence, the maximum likelihood estimate
$\hat{\theta}$ of $\{x_1, x_2, x_3, x_4\}$ is given by 
\begin{equation}\label{eq:closed4}
\hat{\theta} = \frac{x_2x_4 - x_1x_3}{x_4 - x_3 + x_2 - x_1} + \frac{\sqrt{(x_4 - x_3)(x_3 - x_2)(x_4 - x_1)(x_2 - x_1)}}{x_4 - x_3 + x_2 - x_1}i. 
\end{equation} 
The factorization and the computations of discriminants are easily done by using computer algebra systems.

\subsection{The case of $n \ge 5$}
To our knowledge, there are no results for such formulae when $n \geq 5$.
We will show here that the answer is negative,
that is,
there are no algebraic closed formulae for the maximum likelihood estimates
in the following sense:

\begin{Def}
Let $D := \{(x_1, \dots, x_n) \in \R^n: x_1 < \dots < x_n \}$
be the set of all ordered sample of size $n$. 
We call a function $F : D \to \H$ {an algebraic closed-form formula
for the maximum likelihood estimator of samples of size $n$
from the Cauchy distribution} if\\
(i) For every $(x_1, \dots, x_n) \in D$,
	$F(x_1, \dots, x_n)$ is the maximum likelihood estimate of
	observed sample $(x_1, \dots, x_n)$ from the Cauchy distribution. \\
(ii) For every $(x_1, \dots, x_n) \in D \cap \Q^n$,
	there exists a finite sequence of fields
	$\Q = K_0 \subset K_1 \subset \dots \subset K_r$
	such that $F(x_1, \dots, x_n) \in K_r$
	and $K_i = K_{i-1}(\sqrt[n_i]{\alpha_{i}})$
	for some $\alpha_i \in K_{i-1}$,
	where we assume that $\sqrt[n_i]{\alpha_{i}}$
	is a solution of $x^{n_i} - \alpha_i  = 0$ in $\C$. 
\end{Def}

In the above definition, we allow orderings of samples. 
For $n = 3$ or $4$, it is shown in the above subsections that there exists
an algebraic closed-form formula. 
For $n \geq 5$, let us start with showing the main result in this section.

\begin{Thm}
If there exists $(x_1, \dots, x_n) \in D \cap \Q^n$ such that $R_n$ is irreducible over $\Q$ and the Galois group of $R_n$ over $\Q$ is not solvable, then, 
there is no algebraic closed-form formula for the maximum likelihood estimate
of samples of size $n$.  
\end{Thm}

\begin{proof}
Let $\hat{\theta} \in \H$ be the maximum likelihood estimate of
$(x_1, \dots, x_n)$ and
$\overline{\Q(\hat{\theta})}$ be the Galois closure of the field extension $\Q(\hat{\theta}) / \Q$. 
Since $R_n$ is irreducible over $\Q$, the minimal polynomial of $\hat{\theta}$ over $\Q$ is $R_n$. 
Hence, 
$\overline{\Q(\hat{\theta})} = L$, where $L$ denotes the minimal splitting field of $R_n$ over $\Q$. 

If there exists an algebraic closed-form formula, 
then, there exists a finite sequence of fields $\Q = K_0 \subset K_1 \subset \dots \subset K_r$ such that $\hat{\theta} \in K_r$ and $K_i = K_{i-1}(\sqrt[n_i]{\alpha_{i}})$ for some $\alpha_i \in K_{i-1}$ and $n_i \ge 2$, where we assume that  $\sqrt[n_i]{\alpha_{i}}$ is a solution of $x^{n_i} - \alpha_i  = 0$ in $\C$. 
Then, by a standard argument in the Galois theory, we can see that there exists a Galois extension $K_r' / \Q$ such that $K_r \subset K_r'$ and 
 $\mathrm{Gal}(K_r' / \Q)$ is solvable. 
Since $\Q(\hat{\theta}) \subset K_r$ and $L / \Q$ is a Galois extension, 
it holds that 
$L \subset K_r'$. 
Since there exists a surjective homomorphism from $\mathrm{Gal}(K_r' / \Q)$ to $\mathrm{Gal}(L / \Q)$, $\mathrm{Gal}(L / \Q)$ is also solvable. 
However, we can see that $\mathrm{Gal}(L / \Q)$ is not solvable.
\end{proof}

\begin{Cor}
For $n = 5,6,7$, there is no algebraic closed-form formula for the maximum likelihood estimates of samples of size $n$ from the Cauchy distribution. 
\end{Cor}

\begin{proof}
For the following specific choices of samples, 
we will see that $R_n$ is irreducible over $\Q$ and $\mathrm{Gal}(L / \Q)$ is  not solvable for the minimal splitting field $L$ of $R_n$. 
Let $n = 5$.  
The above holds for $(x_1, x_2, x_3, x_4, x_5) = (-3, -1, 2, 3, 4)$. 
Let $n = 6$.  
The above holds for $(x_1, x_2, x_3, x_4, x_5, x_6) = (-3/2, -1/2, 0, 1/5, 4/3, 22/7)$. 
Let $n = 7$. 
The above holds for \\
$(x_1, x_2, x_3, x_4, x_5, x_6, x_7) = (-8,-5,-3, -1, 2,7,10)$. 
\end{proof}

We conjecture that 
there are no algebraic closed-form formulae
for the maximum likelihood estimates
of samples of size $n$ from the Cauchy distribution
for all $n \ge 5$.

\begin{remark}
$\mathrm{Gal}(L / \Q)$ in the above proof can be computed 
by using a computer algebra system, say, Magma.
If $(x_1, x_2, x_3, x_4, x_5) = (-3, -1, 2, 3, 4)$, then,
the following assertions hold. \\
(i) $\left|\mathrm{Aut}(\Q(\hat{\theta})/\Q)\right| \le  \left[\Q(\hat{\theta}) : \Q\right] = \deg(R_5) = 12$.\\
(ii) $[L : \Q] = |\mathrm{Gal}(L/ \Q)| = 46080$.\\
(iii) $R_5 (\z) = 8125\z^{12}-109500\z^{11}+300400\z^{10}+2485600\z^9-19364585\z^8 +41746540\z^7+37375695\z^6-301644350\z^5+341202840\z^4
+365505300\z^3-940185495\z^2+316153530\z+227200140$.
\end{remark}

\begin{remark}
We remark that for some $(x_1, \dots, x_n)$, the irreducibility of $R_n$ or  the non-solvability of the Galois group of $R_n$ fails. 
For example, take $(x_1, x_2, x_3, x_4, x_5) = (-2, -1, 0, 1, 2)$.
In this case, we see that 
\[ 
\begin{split}
R_5 (\z) &= 625\z^{12}-8750\z^{10}+22750\z^8-3625\z^6-31300\z^4+16400\z^2-960\\
 &= 5 (5\z^4 - 5\z^2 - 12) (25\z^8 - 325\z^6 + 645\z^4 - 330\z^2 + 16).
\end{split}
\]
The Galois group of $R_5$ over $\Q$ is solvable and
the maximum likelihood estimate for this sample is given by 
$\hat{\theta} = \sqrt{\frac{\sqrt{53/5}-1}{2}} i$.
\end{remark}

\section{Some properties of the maximum likelihood estimates}\label{sec:properties}

In this section we gather some supplementary facts
which are easily obtained in our setting.

\subsection{Relative position of the maximum likelihood estimates}

As we have already mentioned in the introduction,
one of the most prominent properties of the Cauchy distribution is 
that it allows the existence of outliers.
In this subsection, we are concerned with a relative position of
the maximum likelihood estimate with respect to the
minimum and maximum values of the observed sample.
{\it We assume throughout this subsection that the observed sample
is ordered as $x_1 < x_2 < \dots < x_n$.} 
Recall that $q(\theta) = \theta - n \frac{h(\theta)}{h'(\theta)}$.
We denote by $\overline{B(\theta, r)}$ the closed
ball centered at $\theta\in \C$ and
radius $r$.

\begin{Prop}
Let $\theta \in \C$ and $R := \max\{|\theta - x_j| : 1 \le j \le n\}$. 
Then, 
$q\left(\C \setminus \overline{B(\theta, R)}\right) \subset \overline{B(\theta, R)}$.
\end{Prop}

\begin{proof}
Recall that all zeroes $h(z)$ are contained in $\overline{B(\theta,R)}$.
Take $\zeta$ outside of $\overline{B(\theta,R)}$.
Then Laguerre's separation theorem (see, e.g., \cite[p.~20]{Borwein1995})
asserts that all zeroes of $nh(z) - (z - \zeta) h'(z)$ are
also contained in $\overline{B(\theta,R)}$, 
in other words, for any $\zeta \in \C \setminus \overline{B(\theta,R)}$,
there exists $z \in \overline{B(\theta,R)}$ such that $q(z) = \zeta$.
It means that $\C \setminus \overline{B(\theta,R)}$ is contained in $q(\overline{B(\theta,R)})$.
Since $q$ is bijective, we have the conclusion.
\end{proof}

\begin{Cor}\label{cor:halfcircle}
The maximum likelihood estimate $\hat{\theta}$ of 
the ordered sample
$\{x_1, x_2, \ldots, x_n\}$ satisfies $\left| \hat{\theta} - \frac{x_1 + x_n}{2} \right| \le \frac{x_n - x_1}{2}.$ 
\end{Cor}

\begin{proof}
Consider a closed ball $B \equiv \overline{B(\frac{x_1 + x_n}{2},\frac{x_n - x_1}{2})}$.
If $\theta$ is outside of $B$, $q(\theta) \neq \theta$.
Since the maximum likelihood estimate $\hat{\theta}$ satisfies
$\hat{\theta} = \overline{q(\theta)}$, and the center of $B$ lies on the
real axis, we conclude that $\hat{\theta} \in B$.
\end{proof}

This fact leads us to the following definition. 
\begin{Def}
Let $\hat{\theta} \equiv \hat{\theta}(x_1, x_2, \ldots, x_n)$ be the maximum likelihood estimate of 
the ordered sample $\{x_1, x_2, \ldots, x_n\}$ from the Cauchy distribution.
(i) The relative position of the
	$\hat{\theta}$ in $ \H \cap \D$  is defined by $\xi := \frac{2\hat{\theta} - (x_n + x_1)}{x_n - x_1} \in  \H \cap \D$. \\
(ii) The relative distance of $\hat{\theta}$
	to the boundary is defined by $1 - \frac{|2\hat{\theta} - (x_n + x_1)|}{x_n - x_1}$.
\end{Def}

These notions will be connected with the speeds of the convergences in some numerical schemes in Appendix \ref{sec:numeric}. 
As the following shows, the estimate in Corollary \ref{cor:halfcircle} is best in a sense if $n$ is even. 
\begin{Prop}
(i) Let $n \ge 4$ be an even number. 
For every $\xi \in \H \cap \D$, 
there exists an ordered sample 
$\{x_1 < \ldots < x_n\}$ 
such that its relative position of the maximum likelihood estimate is $\xi$. \\
(ii) Let $n \ge 3$ be an odd number. 
Then, there exists $\epsilon_n > 0$ such that for every $\xi \in \H \cap B(i,\epsilon_n)$, 
there is no ordered sample $\{x_1 < \ldots < x_n\}$ such that its relative position of the maximum likelihood estimate is $\xi$. 
\end{Prop}

\begin{proof}

(i) 
We will construct an ordered sample $\{x_1, \dots, x_n\}$ such that $x_1 = -1$ and $x_n = 1$. 
Let $\xi \in \H$.  
Then, it holds that for some $0 < \psi_1 < \psi_2 < 2\pi$, 
$\left\{\frac{t - \xi}{t - \overline{\xi}} : t \in [-1,1] \right\} = \{e^{i\psi} : \psi \in [\psi_1, \psi_2]\}$.

If $\xi \in \H \cap \D$, then, $-\Re\left(\frac{-1 - \xi}{-1 - \overline{\xi}}\right) \le \Re\left(\frac{1 - \xi}{1 - \overline{\xi}}\right)$ and hence, $\psi_2 - \psi_1 > \pi$. 
Now we can choose distinctive $n$ points $y_1, \dots, y_{n/2}$ and $z_1, \dots, z_{n/2}$ in $[-1,1]$ such that 
$\frac{y_j  - \xi}{y_j - \overline{\xi}} + \frac{z_j  - \xi}{z_j - \overline{\xi}} = 0, \ 1 \le j \le n/2,$ 
and $y_1 = -1$ and $z_{n/2} = 1$. 
Let $(x_1, \dots, x_n) = (y_1, \dots, y_{n/2}, z_1, \dots, z_{n/2})$. 
Then, from Corollary \ref{cor:Cauchy-likelihood-equation4} and $x_1 = y_1 = -1$ and $x_n = z_{n/2} = 1$, 
we see that 
$$\xi =  \hat{\theta}(y_1, \dots, y_{n/2}, z_1, \dots, z_{n/2}) = \frac{2\hat{\theta}(x_1, \dots, x_n) - (x_n + x_1)}{x_n - x_1}.$$

(ii) 
We remark that $\frac{-1 - i}{-1 + i}  = i, \ \frac{1 - i}{1 + i}  = -i,$
and $\left\{\frac{t - i}{t + i} : t \in [-1,1] \right\} = \left\{e^{i\psi} : \psi \in [\pi/2, 3\pi/2]\right\}$. 
Let $\frac{x_j - i}{x_j + i} = \alpha_j + i\beta_j, \ \ 1 \le j \le n$.
Then, by using the fact that $n$ is odd, there exists $\delta_n > 0$ such that 
for each $-1 = x_1 < \dots < x_n = 1$ 
it holds that either 
$\min_{j} \alpha_j \le -\delta_n$ 
or 
$\left|\beta_2 + \dots + \beta_{n-1}\right| \ge \delta_n$.  
Hence, it holds that $\left| \sum_{j=1}^{n} \frac{x_j - i}{x_j + i} \right| \ge \delta_n$ 
 for every $-1 = x_1 < \dots < x_n = 1$.  
Since 
\[  \left| \frac{a - \xi}{a - \overline{\xi}} - \frac{a - w}{a - \overline{w}}\right| \le \frac{2|\theta-w|}{|w|}, \  a \in \R, \xi, w \in \H,\]
we see that 
there exists $\epsilon_n > 0$ such that for every $\xi \in \H \cap B(i,\epsilon_n)$ and every $-1 = x_1 < \dots < x_n = 1$, 
$\left| \sum_{j=1}^{n} \frac{x_j - \xi}{x_j + \xi} \right| \ge \delta_n/2$.
\end{proof}

\subsection{Symmetry of the cumulative distribution function}

Now let us restrict ourselves to the case that $n = 3$ and $4$.
We are concerned with some symmetric properties
of the cumulative distribution function of the Cauchy distribution
whose parameter is the maximum likelihood estimate $\hat{\theta}$.
Using the probability distribution $f$ of \eqref{eq:Cauchy-density},
let us define $F(x;\theta) := \int_{-\infty}^{x} f(y;\theta)\,dy$.

\begin{Prop}
Let $\hat{\theta}$ be the maximum likelihood estimate of an ordered sample 
$\{x_1, x_2, \ldots, x_n\}$. 
Then, \\
(i) For $n=3$, $F(x_3;\hat{\theta}) + F(x_1; \hat{\theta}) = 2F(x_2;\hat{\theta})$.\\
(ii) For $n=4$, $F(x_3;\hat{\theta}) - F(x_1; \hat{\theta}) = F(x_4;\hat{\theta}) - F(x_2;\hat{\theta}) = 1/2$.
\end{Prop}

\begin{proof}
We denote by $\hat{\mu}$ and $\hat{\sigma}$ the real and the imaginary part
of $\hat{\theta}$, respectively.
Let $b_j := \arctan\left(\frac{x_j - \hat{\mu}}{\hat{\sigma}}\right), \ 1 \le j \le n$. 
Since 
\[ \sum_{j=1}^{n} \frac{(x_j - \hat{\mu})^2 - \hat{\sigma}^2}{(x_j - \hat{\mu})^2 + \hat{\sigma}^2} = \sum_{j=1}^{n} \frac{(x_j - \hat{\mu}) \hat{\sigma}}{(x_j - \hat{\mu})^2 + \hat{\sigma}^2} = 0,\]
we can derive that $F(x_j; \hat{\theta}) = b_j, \ 1 \le j \le n$, and 
\begin{equation}\label{eq:sum-zero} 
\exp(2 i b_1) + \dots + \exp(2 i b_n) = 0, 
\end{equation}
where $ \exp(2 i b_1),  \dots, \exp(2 i b_n)$ are placed on the unit circle anticlockwise. 

(i) It suffices to show that $b_1 + b_3 = 2b_2$. 
By a rotation and $b_1 \ne b_3$, 
we assume that $\exp(2 i b_1) + \exp(2 i b_3) \le 0$. 
Then, $b_2 = 0$ and hence $\exp(2 i b_2) = 1$. 
Hence $\exp(2 i b_1) + \exp(2 i b_3) = -1$. 
Hence, 
$\exp(2 i b_3) = \exp(2\pi i /3)$ and $\exp(2 i b_1) = \exp(4\pi i /3)$.
Since $b_j \in (-\pi,\pi)$, we see that $b_3 = 2\pi i /3$ and $b_1 = -2\pi i /3$. 

(ii) 
Since $F(x_{j+2}; \hat{\theta}) - F(x_j; \hat{\theta}) = b_{j+2} - b_j, \ \ j = 1,2,$ 
it suffices to show that $b_{j+2} - b_j = \pi, \ \ j = 1,2$. 
By rotating the unit circle, 
we can assume that $\exp(2 i b_1) + \exp(2 i b_2) \le 0$. 
By \eqref{eq:sum-zero} and $b_1 < b_2 < b_3 < b_4$, 
we can assume that $\exp(2 i b_1) + \exp(2 i b_2) \in (-2, 0)$. 
Let $r := |\exp(2 i b_1) + \exp(2 i b_2)|$. 
Consider the distance $d_r (\eta)$ between $r$ and $\exp(i\eta)$ in the complex plane. 
We see that $d_r (\eta)^2 = r^2 + 1 - 2r\cos \eta$.
It is strictly decreasing on $(-\pi,0)$ and strictly increasing on $(0, \pi)$.  
Since $0 < r < 2$, it holds that $d_r (0) = r-1 < 1 < d_r (\pi) = r+1$.
Hence, there exists exactly one pair of points $(\eta_1, \eta_2)$ such that $\eta_1 +  \eta_2 = 0, -\pi/2 < \eta_1 < 0 < \eta_2 < \pi/2$ and $d_r (\eta_1) = d_r (\eta_2) = 1$. 
By the uniqueness and $b_3 < b_4$, 
we see that $\exp(2 i b_j) = \exp(i\eta_{j-2}) = - \exp(2 i b_{j-2}), \ j = 3,4$.
Since $2b_j \in (-\pi, \pi)$, 
we see that $b_{j+2} - b_j = \pi, \ \ j = 1,2$. 
\end{proof}

{\it Acknowledgments.} \ \ 
The first and second authors were supported by JSPS KAKENHI 19K14549 and 16K05196 respectively. 

\appendix
\section{Numerical computations}\label{sec:numeric}

This section is devoted to some numerical examples. 
Although there are several ways to compute the maximum likelihood estimates,
we give, in this appendix,
some singular examples such that most of those methods do {not} work well. 
Let us begin with summarizing some methods to compute the
maximum likelihood estimates.

First of all, the Newton--Raphson method is standard and classical.
It is used by \cite{Haas1970, Hinkley1978}.  
The EM algorithm, more specifically, the iteratively reweighted least squares method, is also a standard method. 
It is used by \cite{Auslan1995, Arslan1998, Dempster1977, Kent1991}. 
Furthermore, recently, several useful functions in R are provided
for numerical computations of the maximum likelihood estimate
of parametric statistical distributions.
Indeed, the \texttt{nlminb} function in R  was used by \cite{Kravchuk2012}.
This is a quasi Newton--Raphson method. 

Although there are many packages in R which can compute the maximum likelihood estimate of some parametric models, we focus on the ``\texttt{optimx}" and ``\texttt{maxLik}" packages. 
The \texttt{optimx} package contains 14 methods and the \texttt{maxLik} package contains 5 methods. 
Both packages contain the Broyden--Fletcher--Goldfarb--Shanno and the Nelder--Mead methods. 
The \texttt{optimx} package also contains the \texttt{nlminb} function, which is used by \cite{Kravchuk2012}. 
See \cite{Nash2011} for details of the \texttt{optimx} package. 
The \texttt{maxLik} package also contains the Newton--Raphson method and the simulated-annealing method.
See \cite{Henningsen2011} for details of  the \texttt{maxLik} package. 
It is true that these methods work well for non-singular samples and in particular the Newton--Raphson method yields very fast convergence in many cases, however, most of these methods, including the Newton--Raphson method, 
sometimes  diverge for singular samples. 
Examples \ref{exa:singu1} and \ref{exa:singu2} below deal with such cases.   

We can apply some  of the results in the above sections for numerical computations. 
As in Section \ref{sec:fixed}, an iterative scheme is derived from Theorem \ref{main} (iii). 
Contrary to the above case, this definitely converges to the maximum likelihood estimate exponentially fast, however, 
as is pointed out at Remark \ref{rem:may-be-slow},
the convergence could be very slow for singular samples. 
See Examples \ref{exa:singu1} and \ref{exa:singu2}. 

In Section \ref{sec:poly}, we characterized the maximum likelihood estimate
of a sample of size $n$ as a unique root of $R_n$ in $\H$. 
Since there are numerous studies for root-finding algorithms
for polynomials, we may make use of them to find the maximum likelihood
estimates.
They include the Jenkins--Traub method \cite{Jenkins1970}, the Aberth method \cite{Aberth1973}, and the Hirano method \cite{Murota1982} as a few examples.
In this section, we use the \texttt{polyroot} function in R, which uses the Jenkins--Traub algorithm, or the \texttt{Roots} function in Magma. 
One disadvantage of this approach is that $\deg (R_n)$ is large if $n$ is large. 
See Remark \ref{rem:n-is-large} below.

For the Newton--Raphson method and the iteration scheme, it is important to choose suitable starting points.  
Let $x_1 \le \dots \le x_n$ be a sample of size $n$. 
In this paper, we adopt  $\mathrm{median}(x_1, \dots, x_n) + i \mathrm{IQR}(x_1, \dots, x_n)$, 
where IQR is the interquartile range of a sample:
\begin{multline*}
\mathrm{IQR}(x_1, \dots, x_n) := \left( \left( j_2 + 1 - \frac{3n+1}{4} \right) x_{j_2} +  \left( \frac{3n+1}{4} - j_2 \right) x_{j_2 + 1} \right)\\
- \left( \left(  j_1 + 1 - \frac{n+1}{4} \right) x_{j_1} +  \left( \frac{n+1}{4} - j_1  \right) x_{j_1 + 1} \right),
\end{multline*}
where $j_1$ is the integer such that $(n-1)/4 < j_1 < (n+3)/4$ and 
$j_2$ is the integer such that $(3n-1)/4 < j_2 < (3n+3)/4$. 
This is the default setting of the IQR function in R. 
We denote it by $\widetilde{\z}$. 
This order statistical choice is also used in \cite{Hinkley1978}. 

\begin{remark}\label{rem:n-is-large}
If $n$ is large, then, $Q$, and hence $R_n$, could be very complicated, 
so finding roots of $R_n$ would not be a suitable method. 
However, $q$ is relatively simple, and we can use the definition of $Q$, that is,  $Q(\theta) = q(q(\theta))$. 
If $n$ is sufficiently large, then, with high probability, 
$\widetilde{\z}$ above and the one-step estimators of several $\sqrt{n}$-consistent estimators both approximate the maximum likelihood estimator well, 
so it is a candidate of good starting points. 
See \cite{Akaoka2021-2}. 
\end{remark}

The first example is easy to deal with. 

\begin{Exa}
Let $(x_1, x_2, x_3, x_4, x_5, x_6, x_7) = (-8,-5,-3, -1, 2,7,10).$ 
The following table shows the results for $Q^m(\tilde{z})$,
 $m = 3, 4, \ldots$.
We also show the result of the
Newton-Raphson method for the \texttt{maxLik} package after $5$ iterations,
\texttt{nlminb} function in R, and the root of $R_7$ in $\H$.
Thus, the maximum likelihood estimate is almost $-1.404 + 3.909 i$.

\begin{table}[H]
\begin{minipage}{0.47\hsize}
	\[
		\begin{array}{c|c}
		m      & Q^m(\tilde{z})      \\
		\hline
		3      & -1.404858+3.913066i \\
		4      & -1.40443+3.909587i \\
		5      & -1.404389+3.90925i \\
		\geq 6 & -1.404384+3.909214i
		\end{array}
	\]
\end{minipage}
\begin{minipage}{0.47\hsize}
\[
\begin{array}{c|c}
\text{Newton--Raphson} & -1.404 + 3.909i \\
\hline
\mathtt{nlminb} & -1.404 + 3.909i \\
\hline
\text{root of $R_7$} &-1.4043843+3.909214i
\end{array}
\]
\end{minipage}
\end{table}

\end{Exa}

The following two examples are somewhat singular. 

\begin{Exa}\label{exa:singu1}
Let $(x_1, x_2, x_3, x_4) = (-10065, -8678, -6, 0)$.
Then, by the closed-form formula \eqref{eq:closed4}, 
the maximum likelihood estimate is $-43.3525+611.8279i$. 
However, the convergence of our iterative scheme is very slow.
\begin{table}[H]
	\[
		\begin{array}{c|c}
		m      & Q^m(\tilde{z})      \\
		\hline
		100      & -1339.32+3784.915i \\
		\text{1,000} & -197.044+1382.602i \\
		\text{10,000} & -45.0412+625.3293i \\
		\text{100,000} & -43.3525+611.8279i
		\end{array}
	\]
\end{table}

For this example, the methods indicated above other than Nelder--Mead's do not work,
while the Nelder--Mead's output after 187 iterations is $-41.35 + 597.33i$. 
It seems that one of the reasons why many methods do not work well is that the starting point $\widetilde \z$ is badly chosen. 
We see this by the fact that the relative position of the maximum likelihood estimate is $0.9913855+0.1215753i$, which is very close to the boundary of $\H \cap \D$. 
Indeed, the relative distance to the boundary is 0.0012. 
\end{Exa}

\begin{Exa}\label{exa:singu2}
Let $(x_1, x_2, x_3, x_4, x_5, x_6) = (-10000000, -9000000, 0, 1, 10, 100000)$.
We do not have a closed-form formula for the maximum likelihood estimate.
\begin{table}[H]
	\[
		\begin{array}{c|c}
		m              & Q^m(\tilde{z})      \\
		\hline
		\text{1,000}   & 22.278+2360.648i \\
		\text{10,000}  & 6.958+1003.648i \\
		\text{100,000} & 6.7468+971.561i 
		\end{array}
	\]
\end{table}

The Nelder--Mead method after 135 iterations outputs $10.080 + 1401.651i$. 
The methods indicated above other than the Nelder-Mead method does not work well at all. 
By finding roots of the polynomial $R_6$ using
the \texttt{Roots} function of Magma, 
the unique solution of $R_6$ in $\H$ is equal to $ 6.7468+971.5610i$. 
The relative position of the maximum likelihood estimate is $0.9891102+0.0001924i$, and the relative distance to the boundary is 0.0109. 
\end{Exa}

We finally consider the case of 15 observations of the vertical semi-diameter of Venus, which was examined by Rublik \cite{Rublik2001}, see also \cite{Kravchuk2012}.

\begin{Exa}\label{exa:singu3}
Let 
\begin{align*}
(x_1, \dots, x_{15}) = 
	(&-1.4, -0.44, -0.3, -0.24, -0.22, -0.13, -0.05,\\
	 & 0.06,  0.1,  0.18,  0.2,  0.39,  0.48,  0.63,  1.01). 
\end{align*}
Then, the initial value $\tilde{z} = 0.06+0.525i$. 
\begin{table}[H]
	\[
		\begin{array}{c|c}
		m              & Q^m(\tilde{z})      \\
		\hline
		\text{2} & 0.0269864+0.2618557i \\
		\text{4} & 0.0267463+0.2613197i \\
		\text{6, 8, 10} & 0.0267456+0.2613
		\end{array}
	\]
\end{table}
Many but not all methods in the \texttt{maxLik} and \texttt{optimx}
packages work well for this example;
the \texttt{nlminb} function returns $0.02674557+0.2613182i$. 
\end{Exa}

\end{document}